\documentclass{scrartcl}

\usepackage{amsmath,amssymb}
\usepackage{amsthm}
\usepackage{graphicx}
\usepackage{subfigure}

\usepackage[utf8]{inputenc}
\usepackage[T1]{fontenc}
\usepackage{lmodern}

\usepackage{tikz}
\usetikzlibrary{decorations.markings}
\usepackage[arrow, matrix, curve]{xy}

\usepackage{hyperref}

\usepackage{slashed} 

\usepackage[toc,page]{appendix}
\usepackage{placeins}

\usepackage{mathrsfs} 

\numberwithin{equation}{section}

\newtheorem{theorem}{Theorem}[section]
\newtheorem{corollary}[theorem]{Corollary}
\newtheorem{lemma}[theorem]{Lemma}

\theoremstyle{definition}

\newtheorem{remark}[theorem]{Remark}

\newtheorem*{theorem*}{Theorem}

\renewcommand{\dim}{\textup{dim}}
\renewcommand{\ker}{\textup{ker}}

\newcommand{\D}{\slashed{D}}
\newcommand{\SO}{\textup{SO}}
\newcommand{\Spin}{\textup{Spin}}

\begin{document}
\title{Minimal kernels of Dirac operators along maps}
\author{Johannes Wittmann}
\maketitle
\begin{abstract}
Let $M$ be a closed spin manifold and let $N$ be a closed manifold. For maps $f\colon M\to N$ and Riemannian metrics $g$ on $M$ and $h$ on $N$, we consider the Dirac operator $\D^f_{g,h}$ of the twisted Dirac bundle $\Sigma M\otimes_{\mathbb{R}} f^*TN$. To this Dirac operator one can associate an index in $KO^{-\textup{dim}(M)}(\textup{pt})$. If $M$ is $2$-dimensional, one gets a lower bound for the dimension of the kernel of $\D^f_{g,h}$ out of this index. We investigate the question whether this lower bound is obtained for generic tupels $(f,g,h)$.
\end{abstract}
\tableofcontents

\newpage
\section{Introduction}
Let $M$ be a closed (i.e., compact and without boundary) $2$-dimensional spin manifold with a fixed spin structure and let $N$ be a closed manifold. We study the existence and genericness\footnote{The term ``generic'' will be defined rigorously in Remark \ref{rmk}.} of maps $f\colon M\to N$ and Riemannian metrics $g$ on $M$ and $h$ on $N$, such that the kernel of the Dirac operator $\D^f_{g,h}$ of the twisted Dirac bundle $\Sigma M\otimes_{\mathbb{R}}f^*TN$ has quaternionic dimension zero or one. Here, $\Sigma M$ is the usual complex spinor bundle of $M$ and $\D^f_{g,h}$ is called the \textit{Dirac operator along the map $f$}.  

This problem is inherently tied to the vanishing of an index $\textup{ind}_{f^*TN}(M)\in KO^{-\textup{dim}(M)}(\textup{pt})$, see e.g. \cite[eq. (7.24) on p. 151]{LaMi}, which is a generalization of Hitchin's $\alpha$-index \cite[Section 4.2]{Hitchin}. If $M$ is $2$-dimensional, then we have
\[\textup{ind}_{f^*TN}(M)=\left[\dim_\mathbb{H}\ker \D^{f}_{g,h} \right]_{\mathbb{Z}_2} \]
under the isomorphism $KO^{-2}(\textup{pt})\cong \mathbb{Z}_2:=\mathbb{Z}/2\mathbb{Z}$, where $[k]_{\mathbb{Z}_2}$ denotes the class of $k\in\mathbb{Z}$ in $\mathbb{Z}_2$. Note that $\textup{ind}_{f^*TN}(M)$ is independent of the choice of the Riemannian metrics on $M$ and $N$. It is also invariant under homotopies of $f$. However, the index depends on the choice of spin structure on $M$. This means in particular, that
\begin{align}\label{eq1}
	\dim_\mathbb{H}\ker \D^{f}_{g,h}\ge\begin{cases}
	1, &\text{if }\textup{ind}_{\tilde{f}^*TN}(M)\neq 0,\\
	0,& \text{if }\textup{ind}_{\tilde{f}^*TN}(M)= 0,
	\end{cases}
\end{align}
for any $f\colon M \to N$ homotopic to $\tilde{f}$ and any Riemannian metric $g$ on $M$ and $h$ on $N$.

If equality holds in \eqref{eq1}, then we call the kernel of $\D^{f}_{g,h}$ \textit{minimal}. We expect that for generic tupels $(f,g,h)$ the kernel of $\D^{f}_{g,h}$ is minimal. In an investigation for similar results (references are given in the next section) it turned out that often the strategy to prove such a result is the following: for a given $\tilde{f}\colon M\to N$ one has to find a map $f\colon M\to N$ homotopic to $\tilde{f}$ and Riemannian metrics $g$ on $M$ and $h$ on $N$ such that the kernel of $\D^f_{g,h}$ is minimal. Genericness then follows from well known perturbation results. Finding ``enough'' examples for minimal kernels therefore seems to be the crucial part in proving the expectation.

Our first main theorem addresses the existence of tupels $(f,g,h)$ such that the kernel of $\D^f_{g,h}$ is $1$-dimensional, c.f. Theorem \ref{thm 1}. (Examples with $0$-dimensional kernels are easy to construct.) In particular we show that if $\alpha(M)\neq 0$ (we denote by $\alpha(M)$ Hitchin's $\alpha$-index), $N$ is odd-dimensional and orientable, and $\tilde{f}\colon M\to N$ is null-homotopic, then there exists a map $f\colon M\to N$ homotopic to $\tilde{f}$ and Riemannian metrics $g$ on $M$ and $h$ on $N$ s.t. \[\dim_\mathbb{H}\ker\D^f_{g,h}=1.\] Our second main theorem addresses the genericness of minimal kernels, c.f. Theorem \ref{thm 2}.

\subsection{Motivation}
Our motivation to study this problem is twofold. 

On the one hand, there are many results in the literature concerning the genericness of minimal kernels under the presence of an index. In \cite{AmmDahHum} it is shown that for generic metrics, the dimension of the kernel of the (untwisted) Dirac operator is as small as allowed by the index theorem of Atiyah and Singer (on a closed, connected manifold). This fact generalized results in \cite{BaerDahl} and \cite{Maier}. In the latter article the author also considers $\textup{spin}^c$-manifolds. The dependency of the kernel of the twisted Dirac operator, where one twists with hermitian vector bundles, is considered in \cite{Anghel}. Note that we twist with real vector bundles. This is one of the reasons why we were not able to apply the variational approach of \cite{Anghel} and \cite{Maier} to our situation. Another article related to such problems is \cite{Hitchin}.

On the other hand, the existence of maps $f$ with $\dim_{\mathbb{H}}\ker \D^f_{g,h}=1$ has a concrete application to the theory of Dirac-harmonic maps. Dirac-harmonic maps are the critical points of the supersymmetric analog of the classical Dirichlet energy functional. The supersymmetric analog is the underlying functional for the supersymmetric non-linear sigma model in quantum field theory, see e.g. \cite{Dhmaps}, \cite{Dhmapsregt}, \cite[Chapter 10]{RiemGeomandGeomAna}, and \cite[Part 1, Supersolutions, Chapter 3]{super}. The existence of maps $f\colon M\to N$ such that the kernel of $\D^f_{g,h}$ is $1$-dimensional is needed in order that the so-called heat flow for Dirac-harmonic maps, introduced in \cite{ChenJostSunZhu} for manifolds with non-empty boundary, is also well-posed on closed manifolds, c.f. \cite{JW}. 

\subsection*{Acknowledgments}The author would like to thank Bernd Ammann for his ongoing support and many fruitful discussions. The author's work was supported by the DFG Graduiertenkolleg GRK 1692 ``Curvature, Cycles, and Cohomology''.

\section{Notation and preliminaries from spin geometry} 
In this section we introduce notation and recall some basics from spin geometry which will be relevant in the following, e.g. for understanding the precise meaning of our main theorems. For a more detailed introduction to spin geometry we refer to e.g. \cite{LaMi}, \cite{Neu}, \cite{Hij}, \cite{Fri}, and \cite{Roe}

Let $M$ be an oriented $m$-dimensional manifold and denote by $\textup{GL}^+M$ the $GL^+(m)$-principal bundle of oriented frames for $M$. Moreover, we denote by $\theta\colon \widetilde{GL}^+(m)\rightarrow GL^+(m)$ the universal covering for $m\ge 3$ and the connected twofold covering for $m=2$. A \textit{topological spin structure on $M$} is a $\theta$-reduction of $\textup{GL}^+M$, i.e., a topological spin structure on $M$ is a $\widetilde{GL}^+(m)$-principal bundle $\widetilde{\textup{GL}}^+M$ over $M$ together with a twofold covering $\chi\colon\widetilde{\textup{GL}}^+M\rightarrow\textup{GL}^+M$ that commutes with the projections onto $M$ and is compatible with the group actions of the principal bundles. 

Now let $(M,g)$ be an oriented Riemannian manifold and $\textup{SO}(M,g)$ the $SO(m)$-principal bundle of oriented orthonormal frames for $M$. Restricting $\theta$ to the \textit{spin group} given by $\textup{Spin}(m):=\theta^{-1}(SO(m))$, we define a \textit{metric spin structure on $(M,g)$} to be a $\theta|_{\textup{Spin}(m)}$-reduction of $\textup{SO}(M,g)$. Again, this means that a metric spin structure on $M$ is a $\textup{Spin}(m)$-principal bundle $\textup{Spin}(M,g)$ over $M$ together with a twofold covering $\eta\colon \textup{Spin}(M,g)\rightarrow \textup{SO}(M,g)$ that commutes with the projections onto $M$ and is compatible with the group actions of the principal bundles.

Given a topological spin structure $\chi\colon\widetilde{\textup{GL}}^+M\rightarrow\textup{GL}^+M$ on an oriented manifold $M$, every Riemannian metric $g$ on $M$ defines a metric spin structure $$\chi_g\colon \textup{Spin}(M,g)\rightarrow \textup{SO}(M,g)$$ on $(M,g)$ by $\textup{Spin}(M,g):=\widetilde{\textup{GL}}^+M|_{\textup{SO}(M,g)}$. In the following, the term \textit{spin structure} refers to a topological or metric spin structure and it should always be clear from the context which one we mean. A \textit{spin manifold} is an oriented manifold that admits a spin structure. 

On a Riemannian manifold $(M,g)$ with metric spin structure $\eta$, we have the usual Dirac operator $\D=\D_\eta\colon\Gamma(\Sigma M)\to\Gamma(\Sigma M)$ acting on sections of the complex spinor bundle $\Sigma M$. If we are given a map $f\colon M\to N$, where $(N,h)$ is a Riemannian manifold, we define the \textit{Dirac operator along $f$} 
\[\D^f_{g,h}=\D^f_{\eta,h}\colon \Gamma(\Sigma M\otimes_{\mathbb{R}}f^*TN)\to \Gamma(\Sigma M\otimes_{\mathbb{R}}f^*TN)\]
to be the Dirac operator of the twisted Dirac bundle $\Sigma M\otimes_{\mathbb{R}}f^*TN$. In the notation for $\D^f_{g,h}=\D^f_{\eta,h}$ we highlight either the metric $g$ on $M$ or the spin structure $\eta$ on $M$ in the notation, depending on the context. Locally,
\[\slashed{D}_{\eta,h}^f\psi= (\slashed{D}_\eta\psi^i)\otimes s_i + (e_\alpha\cdot \psi^i)\otimes\nabla^{f^*TN}_{e_\alpha}s_i\]
where $\psi=\psi^i\otimes s_i$, the $\psi^i$ are local sections of $\Sigma M$, $(s_i)$ is a local frame of $f^*TN$, $(e_\alpha)$ is a local orthonormal frame of $TM$, and $\nabla^{f^*TN}$ is the pullback of the Levi-Civita connection on $(N,h)$.

\section{Statement of the results}
In this section we state our main results about the existence and genericness of minimal kernels for Dirac operators along maps. We only consider manifolds that are non-empty and smooth.

\begin{theorem}\label{thm 1}Let $M$ be a $2$-dimensional closed manifold and $N$ be a $n$-dimensional closed manifold. Moreover, assume that
	\begin{itemize}
	\item $M$ is connected, oriented, and of positive genus.
	\item $N$ is connected. If $n$ is even, then we assume that $N$ is non-orientable.
	\end{itemize}
Then the following holds:\\

\textbf{Case $n=2$:}\\
Let $h$ be a Riemannian metric on $N$. Then there exists a spin structure $\chi$ on $M$ and a smooth map $f\colon M\to N$ s.t.
\[\dim_\mathbb{H}\ker\D^f_{\chi_g,h}=1\]
for generic Riemannian metrics $g$ on $M$.\\

\textbf{Case $n\ge 2$:}\\
There exists a spin structure $\chi$ on $M$, a smooth map $f\colon M\to N$, and a Riemannian metric $h$ on $N$ s.t.
\[\dim_\mathbb{H}\ker\D^f_{\chi_g,h}=1\]
for generic Riemannian metrics $g$ on $M$.
\end{theorem}

\begin{remark}\ \label{rmk}
	 \begin{enumerate}
	 	\item By ``generic'' we mean $C^\infty$-dense and $C^1$-open. More precisely, a statement $S=S(g)$ holds for generic Riemannian metrics $g$ on $M$, if there exists a subset $\mathcal{G}\subset \textup{Riem}(M)$ of the space of Riemannian metrics on $M$ which is dense in the $C^\infty$-topology, open in the $C^1$-topology, and $S(g)$ is true for every $g\in\mathcal{G}$.
	 	\item The case $n=1$ was not mentioned in the theorem, since for $1$-dimensional $N$ it is not difficult to find examples for $1$-dimensional kernels. If we choose a spin structure on $M$ for which the Dirac operator on $\Sigma M$ has $1$-dimensional kernel and $f$ to be a constant map, then the kernel of $\D^f_{g,h}$ is $1$-dimensional.
	 	\item If $N$ is $2$-dimensional and orientable, or more general even dimensional and spin, then $\textup{ind}_{f^*TN}(M)$ always vanishes \cite[Proposition 10.1]{AmmGin}, hence in this case the kernel of $\D^f_{g,h}$ is never $1$-dimensional.
	 	\item The above theorem gives information about the existence of minimal kernels if $\textup{ind}_{f^*TN}(M)$ does not vanish. If $\textup{ind}_{f^*TN}(M)$ vanishes, examples of minimal kernels are easy to construct. Just take a spin structure on $M$ for which the Dirac operator on $\Sigma M$ has zero dimensional kernel and twist with a constant map.
	 	\item The proof of Theorem \ref{thm 1} is constructive. We will use differences of spin structures to construct maps $M\to S^1$ and then use certain closed geodesics $S^1\to N$ s.t. the composition $M\to S^1\to N$ is the desired map $f$. 	 
	 \end{enumerate}
\end{remark}

From the proof of Theorem \ref{thm 1} we get the following corollary.
\begin{corollary}\label{cormainthm1} Let $M$ be a $2$-dimensional closed connected spin manifold with $\alpha(M)\neq 0$ and let $N$ be an odd-dimensional orientable closed connected manifold. Let $\tilde{f}\colon M\to N$ be null-homotopic. Then there exists a map $f\colon M\to N$ homotopic to $\tilde{f}$ and Riemannian metrics $g$ on $M$ and $h$ on $N$ s.t.
	\[\dim_\mathbb{H}\ker\D^f_{g,h}=1.\]
\end{corollary}

The next theorem addresses the genericness of minimal kernels, assuming their existence.
\begin{theorem}\label{thm 2}
Let $M$ be a $2$-dimensional closed spin manifold with spin structure $\chi$ and let $N$ be an $n$-dimensional closed manifold. Assume that the kernel of $\D^{f}_{\chi_g,h}$ is minimal for some smooth map $f\colon M\to N$ and some Riemannian metrics $g$ on $M$ and $h$ on $N$.

Then the following holds:
	\begin{enumerate}
		\item For generic Riemannian metrics $\tilde{h}$ on $N$ the kernel of $\D ^{f}_{\chi_{g},\tilde{h}}$ is minimal.
		\item For generic Riemannian metrics $\tilde{g}$ on $M$ the kernel of $\D ^{f}_{\chi_{\tilde{g}},h}$ is minimal.
		\item If $h$ is a real analytic Riemannian metric (and $N$ is real analytic), then the kernel of $\D^{\tilde{f}}_{\chi_{g},h}$ is minimal for generic $\tilde{f}\in[f]$, i.e., for a $C^\infty$-dense and $C^1$-open subset of $[f]$. (Here and in the following, $[f]$ denotes the homotopy class of $f\colon M\to N$.)
	\end{enumerate}
\end{theorem}

\section{Differences of spin structures}
In this section we consider differences of spin structures. These are also treated in \cite{DissAmm, Neu} and they are one of the main tools we use to construct the maps $f$ of Theorem \ref{thm 1}.

In this section, we let $M$ be a $m$-dimensional connected spin manifold. Assume we are given a Riemannian metric $g$ on $M$ and two spin structures $\eta^i\colon \Spin(M,g)^i\to\SO(M,g)$, $i=1,2$. 

Then we define the group homomorphism (c.f. \cite[p. 15]{DissAmm})
\begin{align*}
	\delta=\delta_{\eta^1,\eta^2}\colon \pi_1(\SO(M,g))&\to\mathbb{Z}_2,\\
	 [\gamma]&\mapsto\begin{cases}
	1,  & \text{if either } \gamma \text{ lifts to }\textup{Spin}(M,g)^1 \text{ and }\textup{Spin}(M,g)^2\\
	& \text{or it lifts to none of them.}\\
	-1, & \text{if } \gamma \text{ lifts either to }\textup{Spin}(M,g)^1 \text{ or to }\textup{Spin}(M,g)^2.
	\end{cases}
\end{align*}
We call $\delta$ \textit{the difference of the spin structures $\eta^1$ and $\eta^2$}.

The name originates from the following: Let 
\begin{align*}
	\mathcal{M}&:=\{\textup{spin structures on }(M,g)\} / \text{equivalence},\\
	\mathcal{S}&:=\{\textup{index}\le 2\text{ subgroups of }\pi_1(\SO(M,g))\},
\end{align*}
and consider the maps
\begin{align*}
	\Psi\colon \mathcal{M}\to\mathcal{S},\qquad \big(\eta\colon\Spin(M,g)\to\SO(M,g)\big)\mapsto \eta_*(\pi_1(\Spin(M,g))),
\end{align*}
\begin{align*}
\Omega\colon \mathcal{S}\to \textup{Hom}(\pi_1(\SO(M,g)),\mathbb{Z}_2), \qquad H\mapsto \Omega(H),
\end{align*}
where the group homomorphism $\Omega(H)$ is defined by $\ker(\Omega(H))=H$. Then it holds that
\[\delta=\Omega(\Psi(\eta^1))-\Omega(\Psi(\eta^2)).\]
In particular, we have shown the following lemma.
\begin{lemma}
	If $\eta_1$ and $\eta_2$ are not equivalent, then $\delta_{\eta^1,\eta^2}$ is surjective.
\end{lemma}

In the next lemma we show that $\delta$ descends to a group homomorphism $\pi_1(M)\to\mathbb{Z}_2$.
\begin{lemma}\label{lemma descend}There exists a unique group homomorphism $\overline{\delta}\colon\pi_1(M)\to\mathbb{Z}_2$ s.t. the following diagram commutes
	\begin{align*}
		    \xymatrix{
			\pi_1(\SO(M,g)) \ar[dr]_\delta \ar[rr]& &\pi_1(M)\ar[dl]^{\overline{\delta}}\\
			        &\mathbb{Z}_2 &  }
	\end{align*}
	where the horizontal map is induced by the bundle projection $\SO(M,g)\to M$.
\end{lemma}
\begin{proof}
	There exists an exact sequence
	\begin{align*}
	\ldots\to\pi_1(\SO(m)))\xrightarrow{\iota_*}\pi_1(\SO(M,g))\twoheadrightarrow \pi_1(M)\to \underbrace{\pi_0(\SO(m))}_{=\{1\}}\to\ldots
	\end{align*}
	Hence it suffices to show that for every $[\gamma]\in\pi_1(\SO(m))$ we have $[\iota\circ \gamma]\in\ker(\delta)$. This directly follows from the commutative diagram 
		\begin{align*}
	\xymatrix{
		&\Spin(m) \ar@{^{(}->}[r] \ar[d] & \Spin(M,g)^i\ar[d]^{\eta^i}\\
		S^1\ar[r]^\gamma&\SO(m)\ar@{^{(}->}[r]^\iota&\SO(M,g)  }
	\end{align*}
\end{proof}

\subsection{Relation to spin structures and spinor bundles}\label{section relation}
Let $\overline{\delta}=\overline{\delta}_{\eta^1,\eta^2}\colon \pi_1(M)\to\mathbb{Z}_2$ be the group homomorphism of Lemma \ref{lemma descend}. Assume that $\eta^1$ and $\eta^2$ are not equivalent. Then $\overline{\delta}$ is surjective and hence $\ker(\overline{\delta})\subset\pi_1(M)$ is an index $2$ subgroup of $\pi_1(M)$. We denote by 
\[p\colon P\to M\]
the connected twofold covering with $p_*(\pi_1(P))=\ker(\overline{\delta})$.
\begin{lemma}\label{lemma identi spinstr} There exists an isomorphism of $\Spin(m)$-principal bundles
	\begin{align*}
		F\colon \Spin(M,g)^1\times_{\mathbb{Z}_2}P\to \Spin(M,g)^2
	\end{align*}
	where $\mathbb{Z}_2=\ker(\Spin(m)\to\SO(m))$ acts on $\Spin(M,g)^1$ from the right and $\mathbb{Z}_2$ acts on $P$ from the left, and it holds that
	\begin{align}\label{eq 1}
		\eta^2(F([a,e]))=\eta^1(a)
	\end{align}
	for all $a\in\Spin(M,g)^1$, $e\in P$.
\end{lemma}
Due to its technical nature, the proof will be done in the appendix.

For the remainder of this section, let us additionaly assume that $M$ is closed, $2$-dimensional, and of positive genus. In the following we want to relate the associated (complex) spinor bundles
\[\Sigma^iM=\Spin(M,g)^i\times_{\rho}\Sigma_2,\]
$i=1,2$, where $\rho\colon\Spin(2)\to\textup{Aut}(\Sigma_2)$ is the complex spinor representation.

\begin{lemma}\label{lemma induced map}There exists a smooth map 
	\[f=f_\delta\colon M\to S^1\]
such that the following diagram commutes
\begin{align*}
\xymatrix{
	& \pi_1(S^1)=\mathbb{Z}\ar[d]^{x\mapsto [x]}\\
	\pi_1(M)\ar[r]^{\overline{\delta}}\ar[ru]^{f_*}&\mathbb{Z}_2  }
\end{align*}
and $f_*\colon\pi_1(M)\to\pi_1(S^1)$ is surjective.
\end{lemma}
\begin{proof} Since $M$ is a orientable and of genus $g\ge1$ we have $H_1(M;\mathbb{Z}_2)\cong \mathbb{Z}^{2g}$. Combining this with the Hurewicz theorem it is easy to get a commutative diagram
	
\begin{align*}
\xymatrix{
	& \pi_1(S^1)=\mathbb{Z}\ar[d]^{x\mapsto [x]}\\
	\pi_1(M)\ar[r]^{\overline{\delta}}\ar[ru]^{h}&\mathbb{Z}_2  }
\end{align*}	
for a surjective homomorphism $h\colon \pi_1(M)\to \pi_1(S^1)=\mathbb{Z}$. It is a standard result that in this situation there exists a smooth map $f\colon M\to S^1$ such that the induced map $f_*\colon \pi_1(M)\to \pi_1(S^1)$ is given by $f_*=h$. (This follows e.g. from \cite[Proposition 1B.9 on p. 90]{Hatcher}.) We have shown the lemma.
%
%
\end{proof}

\begin{lemma}\label{lemma ident bdls}Let $E\to S^1$ be a Möbius bundle (i.e., $E\to S^1$ is a non-trivial real vector bundle of rank $1$). Then there exists an isomorphism of complex vector bundles
\begin{align*}
	Q\colon\Sigma^1M\otimes_\mathbb{R}f^*E\to \Sigma^2M
\end{align*}
where $f\colon M\to S^1$ is the map from Lemma \ref{lemma descend}, such that
\begin{align}\label{eq 2}
	Q\circ \D^{f^*E}_{\eta^1}=\D_{\eta^2}\circ Q.
\end{align}
Here, $\D_{\eta^2}$ is the usual Dirac operator on the bundle $\Sigma^2M$ (with respect to the spin structure $\eta^2$) and $\D^{f^*E}_{\eta^1}$ is the Dirac operator of the twisted Dirac bundle $\Sigma^1M\otimes_\mathbb{R}f^*E$.
\end{lemma}
\begin{proof}Let $e\colon SE\to S^1$ be the unit sphere bundle of $E\to S^1$ (w.r.t. an arbitrary bundle metric on $E$). Then $e\colon SE\to S^1$ is a non-trivial twofold covering. Hence the pullback $f^*e\colon f^*(SE)\to M$ is also a twofold covering.

\textbf{Step 1:} $f^*e\colon f^*(SE)\to M$ is a connected covering with $(f^*e)_*\pi_1(f^*(SE))=\ker(\overline{\delta})$.\footnote{In the notation of the beginning of Section \ref{section relation} this means that $f^*(SE)$ and $P$ are isomorphic as coverings.}\\
\textit{Proof of step 1:} Since $M$ and $S^1$ are connected, the induced map $f_*\colon \pi_0(M)\to\pi_0(S^1)$ is bijective. Moreover, $f_*\colon \pi_1(M)\to\pi_1(S^1)$ is surjective by Lemma \ref{lemma descend}. Then it follows from covering space theory that $f^*e\colon f^*(SE)\to M$ is connected. (See e.g. \cite[Lemma 3.6]{pullback}.) Moreover, it holds that
\begin{align*}
	(f^*e)_*\pi_1(f^*(SE))&=(f_*)^{-1}(e_*\pi_1(SE))\\
	&=(f_*)^{-1}(2\mathbb{Z})\\
	&=\ker{\overline{\delta}},
\end{align*}
where the first equality again follows from covering space theory (see e.g. \cite[Lemma 3.4]{pullback}) and the third equality follows directly from the commutative diagram in Lemma \ref{lemma induced map}.\hfill\checkmark

\textbf{Step 2:} The map
\begin{align*}
	\alpha\colon \left(\Spin(M,g)^1\times_{\mathbb{Z}_2} f^*(SE)\right)\times_\rho\Sigma_2&\to \left(\Spin(M,g)^1\times_\rho\Sigma_2\right)\otimes_{\mathbb{R}}f^*E,\\
	[[a,b],v]&\mapsto [a,v]\otimes b,
\end{align*}
where $a\in\Spin(M,g)^1$, $b\in f^*(SE)$, $v\in\Sigma_2$, is an isomorphism of complex vector bundles.\\
\textit{Proof of step 2:} Note that $\Spin(M,g)^1\times_{\mathbb{Z}_2} f^*(SE)$ is a $(\Spin(2)\times_{\mathbb{Z}_2}\mathbb{Z}_2\cong\Spin(2))$-principal bundle, hence the source of the map $\alpha$ is well-defined. It is not difficult to verify that $\alpha$ is well-defined. The (well-defined) inverse of $\alpha$ is given on elementary tensors by
\begin{align*}
	\alpha^{-1}([a,v]\otimes \tilde{b})=\left[\left[a,\frac{\tilde{b}}{\|\tilde{b}\|}\right],\|\tilde{b}\|v\right]
\end{align*} 
where $a\in\Spin(M,g)^1$, $v\in\Sigma_2$, $\tilde{b}\in f^*E$, $b\neq 0$. We have shown step 2.\hfill\checkmark

Combining both steps with Lemma \ref{lemma identi spinstr} we get
\begin{align*}
	\Sigma^1M\otimes_{\mathbb{R}}f^*E&=\left(\Spin(M,g)^1\times_\rho\Sigma_2\right)\otimes_{\mathbb{R}}f^*E\\
	&\cong\left(\Spin(M,g)^1\times_{\mathbb{Z}_2}f^*(SE)\right)\times_\rho\Sigma_2\\
	&\cong\Spin(M,g)^2\times_\rho \Sigma_2\\
	&=\Sigma^2M,
\end{align*}
i.e., we define $Q:=(F,\textup{id}_{\Sigma_2})\circ\alpha^{-1}$. Using the construction of $\alpha$ and equation \eqref{eq 1} one readily checks that $Q$ commutes with Clifford-multiplications on $\Sigma^1M\otimes_{\mathbb{R}}f^*E$ and $\Sigma^2 M$. Combining this with the local formulas for the covariant derivatives on the spinor bundles $\Sigma^iM$ it is straightforward to deduce \eqref{eq 2}.

\end{proof}

\section{Proof of the main theorems}
Before we come to the proof of the main theorems we need one more lemma.
\begin{lemma}\label{lemma 1}Let $M$ be a $2$-dimensional closed connected spin manifold of positive genus. Then there exist spin structures $\chi^i$ on $M$, $i=1,2$, such that
	\begin{align}\label{eq 4}
		\begin{split}
		\dim_\mathbb{H}\ker(\D_{\chi^1_g})&=1,\\
		\dim_\mathbb{H}\ker(\D_{\chi^2_g})&=0,
		\end{split}
	\end{align}
for generic Riemannian metrics $g$ on $M$.
\end{lemma}
\begin{proof}Since $M$ has positive genus, there exist spin structures $\chi^1$ and $\chi^2$ on $M$ such that $\alpha(M,\chi_1)=0$ and $\alpha(M,\chi_2)\neq 0$. (We denote by $\alpha(M,\chi)$ Hitchin's $\alpha$-index.) Now we apply \cite[Theorem 1.1]{AmmDahHum} for $\chi_1$ and $\chi_2$. The lemma follows since the intersection of two open and dense sets is again open and dense.
\end{proof}

\begin{proof}[Proof of Theorem \ref{thm 1}, case $n=2$] Let $h$ be any Riemannian metric on $N$. We choose spin structures $\chi^1,\chi^2$ on $M$ and a $C^\infty$-dense and $C^1$-open set $\mathcal{G}\subset\textup{Riem}(M)$ s.t. \eqref{eq 4} holds for every $g\in\mathcal{G}$. Let $g\in\mathcal{G}$ be arbitrary and let $f=f_\delta\colon M\to S^1$ be the map of Lemma \ref{lemma induced map} where $\delta=\delta_{\chi^1_g,\chi^2_g}$.
	
	Since $N$ is non-orientable, there exists an orientation reversing simple (i.e., without self-intersections) closed geodesic $\gamma\colon S^1\to N$.\footnote{Recall that a \textit{closed curve} is a smooth map $S^1\to N$ and a \textit{closed geodesic} is a closed curve that is also a geodesic.} A proof of this fact can be found in the appendix. \label{page}Viewing $S^1$ as a submanifold of $N$ via $\gamma$, we have that
	\begin{align*}
	\gamma^*TN\cong TS^1 \oplus (TS^1)^\bot
	\end{align*}
	where $TS^1\cong\underline{\mathbb{R}}$ is trivial and $(TS^1)^\bot \to S^1$ is non-trivial (i.e., a Möbius bundle), since $\gamma$ is orientation reversing. Since $\gamma$ is a geodesic, we have
	\begin{align}\label{eq 3}
	\nabla^{\gamma^*TN}\cong \begin{pmatrix}
	\nabla^{TS^1} &0\\
	0 & \nabla^{(TS^1)^\bot}
	\end{pmatrix}
	\end{align}
	under the above isomorphism. (Here, $\nabla^{\gamma^*TN}$ is the pullback of the Levi-Civita connection on $(N,h)$ along $\gamma$, and $\nabla^{TS^1}$ and $\nabla^{(TS^1)^\bot}$ are the projections of $\nabla^{\gamma^*TN}$ on $TS^1$ and $(TS^1)^\bot$, respectively.)  We set $\tilde{f}:=\gamma\circ f\colon M\to N$. Applying Lemma \ref{lemma ident bdls} we find that
	\begin{align*}
	\Sigma^1 M\otimes_\mathbb{R} \tilde{f}^*TN&\cong \Sigma^1M\otimes_{\mathbb{R}}f^*(\gamma^*TN)\\
	&\cong \Sigma^1M\otimes_{\mathbb{R}}f^*(\underline{\mathbb{R}}\oplus (TS^1)^\bot)\\
	&\cong \left(\Sigma^1M\otimes_{\mathbb{R}}\underline{\mathbb{R}}\right)\oplus \left(\Sigma^1 M\otimes_{\mathbb{R}}f^*(TS^1)^\bot\right)\\
	&\cong \Sigma^1M\oplus\Sigma^2M.
	\end{align*}
	Using \eqref{eq 3} it follows that under this isomorphism it holds that
	\begin{align}\label{eq 5}
	\D^{\tilde{f}}_{\chi^1_g,h}\cong \begin{pmatrix}
	\D_{\chi^1_g} &0\\
	0 & \D_{\chi^2_g}
	\end{pmatrix}.
	\end{align}
	In particular,
	\[\ker(\D^{\tilde{f}}_{\chi^1_g,h})\cong\ker(\D_{\chi^1_g})\oplus\ker(\D_{\chi^2_g}).\]
	We conclude by using \eqref{eq 4}.
\end{proof}

\begin{proof}[Proof of Theorem \ref{thm 1}, case $n\ge 2$]We choose spin structures $\chi^1,\chi^2$ on $M$ and $\mathcal{G}\subset\textup{Riem}(M)$ as before. Let $g\in\mathcal{G}$ be arbitrary and let $f=f_\delta\colon M\to S^1$ be the map of Lemma \ref{lemma induced map} where $\delta=\delta_{\chi^1_g,\chi^2_g}$.
	
	Let $h_0$ be a Riemannian metric on $N$ s.t. there exists a simple closed geodesic $\gamma\colon S^1\to N$. \footnote{Given any injective closed immersed curve $\gamma\colon S^1 \to N$, it is not hard to construct a Riemannian metric on $N$ for which $\gamma$ is a simple closed geodesic. One can do this e.g. by using a tubular neighborhood of the image of $\gamma$.} Again, we view $S^1$ as a submanifold of $N$ via $\gamma$.
	
	In the case $n=2$, the key ingredient was the identification \eqref{eq 5}, which followed from \eqref{eq 3}. If the dimension of $N$ is greater than two, it is more complicated to deal with the complement $(TS^1)^\bot\subset TN$ in order to get a suitable higher dimensional analog of $\eqref{eq 3}$. For this reason we will modify the Riemannian metric $h_0$ in a neighborhood of $S^1\subset N$. To that end, let
	\[U_\varepsilon:= \textup{exp}^\bot\{(p,v)\in TN\text{ }|\text{ }p\in S^1,\text{ } v\in(T_pS^1)^\bot, \text{ }\|v\|_{h_0}<\varepsilon\}\]
	be a tubular neighborhood of $S^1$ in $N$, where $\varepsilon>0$ is sufficiently small.
	
	Moreover, let $(\nu_1,\ldots,\nu_{n-1})$ be an orthonormal basis of $(T_{\gamma(0)}S^1)^\bot$ where we think of $S^1$ as $[0,2\pi]$ with $0$ and $2\pi$ identified. We define
	\[\nu_i(t):=P^\gamma_{0,t}\nu_i\]
	where $P^\gamma_{0,t}$ denotes the parallel transport in $(N,h_0)$ along $\gamma|_{[0,t]}$ from $\gamma(0)$ to $\gamma(t)$, $t\in[0,2\pi]$. Since $\gamma$ is a geodesic, $(\nu_1(t),\ldots,\nu_{n-1}(t))$ is an orthonormal basis of $(T_{\gamma(t)}S^1)^\bot$ for all $t\in[0,2\pi]$. In the basis $(\nu_1,\ldots,\nu_{n-1})$ the map
	\[P^\gamma_{0,2\pi}\colon (T_{\gamma(0)}S^1)^\bot\to (T_{\gamma(2\pi)}S^1)^\bot\]
	is given by a matrix $A\in\textup{O}(n-1)$. Then we have a diffeomorphism
	\begin{align*}
	T_A:=^{[0,2\pi]\times B_\varepsilon(0)}/_{(0,x)\sim (2\pi,Ax)}&\to U_\varepsilon,\\
	[(t,\sum_{i=1}^{n-1}x_ie_i)]&\mapsto \textup{exp}(\gamma(t),\sum_{i=1}^{n-1}x_i\nu_i(t)),
	\end{align*}
	where $B_\varepsilon(0)\subset\mathbb{R}^{n-1}$ is the open ball of radius $\varepsilon$ with center $0$ and $(e_1,\ldots,e_{n-1})$ is the standard basis of $\mathbb{R}^{n-1}$.
	
	Note that if $A,B\in \textup{O}(n-1)$ are in the same connected component of $\textup{O}(n-1)$, then $T_A$ and $T_B$ are diffeomorphic. We will use this statement a few times below without further mentioning it.
	
	Let $A\in\textup{O}(n-1)$. If we endow $T_A$ with the quotient metric induced from the product metric on $[0,2\pi]\times B_\varepsilon(0)$, then the parallel transport in $T_A$ along the curve $c(t):=[(t,0)]$ from $c(0)$ to $c(2\pi)$ is given by
	\[P^c_{0,2\pi}\colon T_{c(0)}T_A\to T_{c(2\pi)}T_A,\qquad P^c_{0,2\pi}=\begin{pmatrix}
	1 & 0\\
	0 & A
	\end{pmatrix}.  \]
	(with respect to the splitting $T_{c(0)}T_A=T_{c(0)}S^1\oplus (T_{c(0)}S^1)^\bot$ where similar as before we write $S^1$ for the image of $c$).
	
	Now we distinguish three cases.\\
	\textbf{Case 1:} $n$ is even and $N$ is non-orientable: 
	
	We can choose $\gamma$ to be orientation reversing (c.f. Lemma \ref{lemma app1}). Then $P^\gamma_{0,2\pi}\colon (T_{\gamma(0)}S^1)^\bot\to (T_{\gamma(2\pi)}S^1)^\bot$ is orientation reversing and hence the associated matrix is an element of $\textup{O}(n-1)\setminus \SO(n-1)$. Therefore,
	\begin{align*}
	U_\varepsilon\cong T_{-I_{n-1}}
	\end{align*}
	where
	\begin{align*}
	I_{n-1}=\begin{pmatrix}
	1 & & \\
	& \ddots & \\
	& & 1
	\end{pmatrix}.
	\end{align*}
	From the discussion above we see that we can choose a Riemannian metric on $U_\varepsilon$ such that
	
	\[P^\gamma_{0,2\pi}\colon (T_{\gamma(0)}S^1)^\bot\to (T_{\gamma(2\pi)}S^1)^\bot,\qquad P^\gamma_{0,2\pi}v=-v\]
	is minus the identity. Using a partition of unity we have shown that there exists a Riemannian metric $h$ on $N$ such that $P^\gamma_{0,2\pi}\colon(T_{\gamma(0)}S^1)^\bot\to (T_{\gamma(2\pi)}S^1)^\bot$ is minus the identity. This means in particular that
	\[\gamma^*TN\cong TS^1\oplus (TS^1)^\bot\cong TS^1\oplus E_1\oplus\ldots\oplus E_{n-1}\]
	where each $E_i\to S^1$ is a Möbius bundle. Moreover, under this isomorphism we have
	\begin{align}
	\nabla^{\gamma^*TN}\cong\begin{pmatrix}\nabla^{TS^1} & & &\\
	& \nabla^{E_1} & &\\
	& & \ddots &\\
	& & & \nabla^{E_{n-1}}
	\end{pmatrix}
	\end{align}
	where $\nabla^{\gamma^*TN}$ is the pullback of the Levi-Civita connection on $(N,h)$ along $\gamma$, and $\nabla^{TS^1}$, $\nabla^{E_i}$ are the projections of $\nabla^{\gamma^*TN}$. Setting
	\[\tilde{f}:=\gamma\circ f\]
	and using Lemma \ref{lemma ident bdls} we get 
	\begin{align*}
	\Sigma^1M\otimes_\mathbb{R}\tilde{f}^*TN&\cong \Sigma^1M\otimes_{\mathbb{R}}(f^*(\gamma^*TN))\\
	&\cong \Sigma^1M\otimes_{\mathbb{R}}\left(\underline{\mathbb{R}}\oplus f^*(E_1)\oplus\ldots\oplus f^*(E_{n-1})\right)\\
	&\cong \Sigma^1M\oplus \Sigma^2M\oplus\ldots\oplus\Sigma^2M
	\end{align*}
	and, similar to the proof of the case $n=2$, under this isomorphism we have
	\begin{align}
	\D^{\tilde{f}}_{\chi^1_g,h}\cong\begin{pmatrix}\D_{\chi^1_g} & & &\\
	& \D_{\chi^2_g} & &\\
	& & \ddots &\\
	& & & \D_{\chi^2_g}
	\end{pmatrix}
	\end{align}
	and therefore
	\[\ker(\D^{\tilde{f}}_{\chi^1_g,h})\cong\ker(\D_{\chi^1_g})\oplus\ker(\D_{\chi^2_g})\oplus\ldots\oplus\ker(\D_{\chi^2_g}).\]
	We conclude by using \eqref{eq 4}.\\
	\textbf{Case 2:} $n$ is odd and $N$ is orientable:
	
	Then $\gamma$ is orientation preserving, hence $P^\gamma_{0,2\pi}\colon (T_{\gamma(0)}S^1)^\bot\to (T_{\gamma(2\pi)}S^1)^\bot$ is orientation preserving and the associated matrix is an element of $\SO(n-1)$. We get
	\[U_\varepsilon\cong T_{-I_{n-1}}\]
	since $-I_{n-1}$ is in the same connected component as the associated matrix (because both are orientation preserving). Now we can proceed as in case 1.\\
	\textbf{Case 3:} $n$ is odd and $N$ is non-orientable:
	
	Again we can assume that $\gamma$ is orientation reversing. Then the tubular neighborhood $U_\varepsilon$ is diffeomorphic to $T_A$ for
	\begin{align*}
	A=\begin{pmatrix}-1 & & &\\
	& 1 & &\\
	& & \ddots &\\
	& & & 1
	\end{pmatrix}.
	\end{align*}
	Then we can proceed analogous to the first two cases, but we have to switch the roles of the spin structures $\chi^1$ and $\chi^2$.
\end{proof}

\begin{proof}[Proof of Theorem \ref{thm 2}]
\textit{Proof of i):} Let $\tilde{h}$ be an arbitrary Riemannian metric on $N$. We define $h_t:=th+(1-t)\tilde{h}$. Then $\D^{f}_{\chi_{g},h_t}$ depends analytically on $t$ in the sense of \cite[Section 11]{Maier}. By \cite[Proposition 11.4]{Maier} the set
\[\{t\in[0,1]\mid \text{the kernel of }\D^{f}_{\chi_{g},h_t} \text{ is not minimal}\}\]
is finite. Hence the set of Riemannian metrics $h$ on $N$ such that the kernel of $\D^{f}_{\chi_{g},h}$ is minimal is $C^\infty$-dense in $\textup{Riem}(N)$. Moreover, it is $C^1$-open.\footnote{One way to see this is to use the Min-Max principle to show that the map $\textup{Riem}(N)\to \mathbb{N}$, $h\mapsto\dim_{\mathbb{H}}\ker \D^{f}_{\chi_{g},h}$, is upper semicontinuous where on $\textup{Riem}(N)$ we choose the $C^1$-topology.}

\textit{Proof of ii):} The proof is similar to the proof of i), i.e., we use linear interpolation and \cite[Proposition 11.4]{Maier}. Note however, that if we vary the metric on $M$, then the space on which the Dirac operators are defined, also changes and we cannot apply the proposition directly. To get rid of this, we identify the spinor bundles on $M$ as in \cite[Section 2.2]{Maier}, \cite{BG} and then we are able to apply the proposition (compare also \cite[Proof of Proposition 3.1]{Maier}).

\textit{Proof of iii):} We want to use the same strategy as before. The difficulty this time is to find a (piecewise) real analytic path between two homotopic elements in $C^\infty(M,N)$, since linear interpolation does no longer work. Let $\tilde{f}\in C^\infty(M,N)$ be any map with $d^N(f(x),\tilde{f}(x))<\frac12 \textup{inj}(N)$ for all $x\in M$, where $\textup{inj}(N)$ denotes the injectivity radius of $N$. We define
\[f_t(x):=\textup{exp}_{\tilde{f}(x)}\left(t\textup{exp}_{\tilde{f}(x)}^{-1}f(x)\right),\]
$x\in M$, where $\exp$ denotes the exponential map of $N$.\footnote{Conceptually, we take a chart of the manifold $C^\infty(M,N)$ around $\tilde{f}$, linearly interpolate between $\tilde{f}$ and $f$ in the chart, and then use the inverse of the chart to go back to $C^\infty(M,N)$. The result is the map $f_t$.} Then we claim that for all but finitely many $t\in[0,1]$ it holds that the kernel of $\D^{f_t}_{\chi_g,h}$ is minimal. To see this, we denote by $P^t\colon T_{f(x)}N\to T_{f_t(x)}N$ the parallel transport along unique shortest geodesics of $N$ joining $f(x)$ and $f_t(x)$ and consider
\[D_t:=(P^t)^{-1}\circ\D^{f_t}_{\chi_g,h}\circ P^t.\]
The claim follows since the family of operators $D_t$ depends analytically on $t$. ($P^t$ depends analytically on $t$ because of the analytic dependence of solutions of ordinary differential equations on parameters. $f_t$ depends analytically on $t$, since the Riemannian metric on $N$ is real analytic.)

Now let $\tilde{f}\in[f]$ be homotopic to $f$ and let $H$ be any homotopy between $f$ and $\tilde{f}$. We view $H$ as a path $H\colon [0,1]\to C^\infty(M,N)$ with $H(0)=f$ and $H(1)=\tilde{f}$. We can cover the image of $H$ by finitely many $C^0$-balls $U_i$ of radius less than $\frac12\textup{inj}(N)$, $i=1,\ldots,N$, such that $U_i\cap U_{i+1}\neq \varnothing$ for $i=1,\ldots,N-1$, and $f\in U_1$, $\tilde{f}\in U_{N}$.

We choose $f_1\in U_1\cap U_2$ arbitrarily. From the beginning of the proof of iii), we get that there exists a homotopy $H^1$ between $f$ and $f_1$ such that the kernel of $\D^{H^1_t}_{\chi_g,h}$ is minimal for all but finitely many $t\in[0,1]$. We can assume that the kernel of $\D^{f_1}_{\chi_g,h}$ is minimal. Continuing in that manner, we conclude that there exists $f_{N-1}\in U_{N-1}\cap U_N$ such that the kernel of $\D^{f_{N-1}}_{\chi_g,h}$ is minimal and a homotopy $H^{N-1}$ between $f_{N-1}$ and $\tilde{f}$ such that the kernel of $\D^{H^{N-1}_t}_{\chi_g,h}$ is minimal for all but finitely many $t\in[0,1]$. Hence the set of maps $f\in[\tilde{f}]$ such that the kernel of $\D^{f}_{\chi_g,h}$ is minimal is $C^\infty$-dense in $[f]$. As above, it is also $C^1$-open.
\end{proof}

\newpage
\begin{appendices}
	\section{Proof of Lemma \ref{lemma identi spinstr}}
	Let us choose $x_0\in M$ and $y_0\in p^{-1}(x_0)$. Then we define a mapping
	\begin{align*}
	F\colon \Spin(M,g)^1\times_{\mathbb{Z}_2}P\to \Spin(M,g)^2
	\end{align*}
	as follows. 
	
	Let $a\in\Spin(M,g)^1_x$ and $b\in P_x$ be given.
	\begin{enumerate}
		\item[1)] Choose a path $\omega\colon [0,1]\to M$ s.t. $\omega(0)=x_0$ and $\omega(1)=x$. Moreover, denote by $\gamma^\omega\colon[0,1]\to P$ the lift of $\omega$ with $\gamma^\omega(0)=y_0$.
		\item[2)] Choose a lift $\tilde{\omega}\colon [0,1]\to\SO(M,g)$ of $\omega$. 
		\item[3)] Choose lifts $\gamma_i^{\tilde{\omega}}\colon [0,1]\to\Spin(M,g)^i$ of $\tilde{\omega}$ satisfying
		\[\gamma_1^{\tilde{\omega}}(0)\cong\gamma_2^{\tilde{\omega}}(0)\]
		where we identify $\Spin(M,g)^1_{x_0}\cong\Spin(M,g)^2_{x_0}$ with a fixed isomorphism (we fix the isomorphism for the whole proof).
		\item[4)] Let $A=A^{\tilde{\omega}}\in\Spin(m)$ and $B=B^\omega\in\mathbb{Z}_2$ be the uniquely determined elements of $\Spin(m)$ and $\mathbb{Z}_2$, respectively, s.t.
		\begin{align*}
			\gamma_1^{\tilde{\omega}}(1)\cdot A=a,\\
			\gamma^{\omega}(1)\cdot B=b.
		\end{align*}
	\end{enumerate}
	Then we define
	\begin{align*}
		F([a,b]):=\gamma_2^{\tilde{\omega}}(1)\cdot A\cdot B.
	\end{align*}
	The main task is to show that $F$ is well-defined, i.e., doesn't depend on the choices made in 1)-3). 
	
	One easily verifys that the definition of $F$ is independent of the choice of the $\gamma_i^{\tilde{\omega}}$, since for each $i=1,2$ there exist exactly two such lifts and they differ only by $-1\in\mathbb{Z}_2=\ker(\Spin(m)\to\SO(m))$.

	Therefore, it remains to show that $F$ is independent of the choice of $\omega$ and $\tilde{\omega}$ in 1) and 2). To that end, we will show the following lemma.
	\begin{lemma}\label{lemma1}
		Choose $x\in M$ and let $\omega,\sigma\colon[0,1]\to M$ be paths from $x_0$ to $x$. Moreover, let $\tilde{\omega}, \tilde{\sigma}\colon [0,1]\to\SO(M,g)$ be lifts of $\omega$ and $\sigma$, respectively. Then the following holds:
		\begin{enumerate}
			\item If $\omega\ast\overline{\sigma}\in\ker(\overline{\delta})$, then $\gamma_2^{\tilde{\omega}}(1)\cdot A^{\tilde{\omega}}=\gamma_2^{\tilde{\sigma}}(1)\cdot A^{\tilde{\sigma}}$ and $B^\omega=B^\sigma$.
			\item If $\omega\ast\overline{\sigma}\notin\ker(\overline{\delta})$, then $\gamma_2^{\tilde{\omega}}(1)\cdot A^{\tilde{\omega}}=\gamma_2^{\tilde{\sigma}}(1)\cdot A^{\tilde{\sigma}}\cdot(-1)$ and $B^\omega=B^\sigma\cdot (-1)$.
		\end{enumerate}
	In particular, $F$ is well-defined.
	\end{lemma}
\begin{proof} Let us prove i) first. Notice that since $\omega\ast\overline{\sigma}\in\ker(\overline{\delta})=p_*(\pi_1(P,y_0))$ (c.f. the beginning of Section \ref{section relation}) we have that $\omega\ast\overline{\sigma}$ can be lifted to a loop in $P$, and we directly get $B^\omega=B^\sigma$. Now we proceed in several steps.\\

\textbf{Step 1:} The assertion of i) holds if $\tilde{\omega}(0)=\tilde{\sigma}(0)$ and $\tilde{\omega}(1)=\tilde{\sigma}(1)$.\\
Since $\omega\ast\overline{\sigma}\in\ker(\overline{\delta})$ it follows from Lemma \ref{lemma descend} that $\alpha:=\tilde{\omega}\ast\overline{\tilde{\sigma}}\in\ker(\delta)$. We further distinguish two cases.\\
\underline{Case 1:} $\alpha$ lifts to a loop in $\textup{Spin}(M,g)^i$, $i=1,2$.\\
In this case we get lifts $\gamma_i^{\tilde{\omega}},\gamma_i^{\tilde{\sigma}}\colon [0,1]\to\Spin(M,g)^i$ of $\tilde{\omega}$ and $\tilde{\sigma}$, respectively, s.t.
\begin{align*}
	\gamma_i^{\tilde{\omega}}(0)&=\gamma_i^{\tilde{\sigma}}(0)\in\Spin(M,g)^i_{x_0},\\
	\gamma_i^{\tilde{\omega}}(1)&=\gamma_i^{\tilde{\sigma}}(1)\in\Spin(M,g)^i_x,
\end{align*}
$i=1,2$ and step 1 directly follows. (Note that we already remarked that the definition of $F$ is independent of the choices in 4).)\\
\underline{Case 2:}  $\alpha$ does not lift to a loop in $\textup{Spin}(M,g)^i$, $i=1,2$.\\
In this case we get lifts $\gamma_1^{\tilde{\omega}},\gamma_1^{\tilde{\sigma}}\colon [0,1]\to\Spin(M,g)^1$ of $\tilde{\omega}$ and $\tilde{\sigma}$, respectively, s.t. $\gamma_1^{\tilde{\omega}}(1)=\gamma_1^{\tilde{\sigma}}(1)$ and $\gamma_1^{\tilde{\omega}}(0)\neq\gamma_1^{\tilde{\sigma}}(0)\in\Spin(M,g)^1_{x_0}$, i.e., 
\begin{align}\label{eq 8}\{\gamma_1^{\tilde{\omega}}(0),\gamma_1^{\tilde{\sigma}}(0)\}=\Spin(M,g)^1_{x_0}.\end{align}
Then we lift $\alpha$ to a path in $\Spin(M,g)^2$ with starting point $\gamma_1^{\tilde{\omega}}(0)$ and this lift gives us a choice for $\gamma_2^{\tilde{\omega}}$ and $\gamma_2^{\tilde{\sigma}}$. We have
\[\gamma_2^{\tilde{\omega}}(0)\neq\gamma_2^{\tilde{\sigma}}(0)\in\Spin(M,g)^1_{x_0}\cong\Spin(M,g)^2_{x_0}\]
and $\gamma_2^{\tilde{\omega}}(0)\cong\gamma_1^{\tilde{\omega}}(0)$. Combining with \eqref{eq 8} we get $\gamma_2^{\tilde{\omega}}(0)\cong\gamma_1^{\tilde{\sigma}}(0)$ and we have shown step 1.\\

\textbf{Step 2:} The assertion of i) holds if $\tilde{\omega}(1)=\tilde{\sigma}(1)$.\\
Let $c\colon[0,1]\to\SO(M,g)_{x_0}$ be a path with $c(0)=\tilde{\omega}(0)$ and $c(1)=\tilde{\sigma}(0)$. Let $\hat{c}$ be the lift of $c$ to $\Spin(M,g)^1$ with $\hat{c}(1)=\gamma_1^{\tilde{\sigma}}(0)$. Note that $\hat{c}$ only takes values in $\Spin(M,g)^1_{x_0}\cong\Spin(M,g)^2_{x_0}$, so we also think of $\hat{c}$ as lift of $c$ to $\Spin(M,g)^2_{x_0}$. Now we can apply the result of step 1 to $\sigma_1:=\sigma\ast x_0$ (where $x_0$ denotes the constant path), $\omega$, $\widetilde{\sigma_1}:=\tilde{\sigma}\ast c$, and $\tilde{\omega}$.\\

\textbf{Step 3:} The assertion of i) holds.\\
Choose $X\in\SO(m)$ such that $\tilde{\sigma}(1)\cdot B=\tilde{\omega}(1)$. We conclude by using step 2.\\

For ii), we first observe the following: if $\omega\ast\overline{\sigma}\notin\ker(\overline{\delta})=p_*(\pi_1(P,y_0))$, then $\omega\ast\overline{\sigma}$ does not lift to a loop in $P$. From this we easily get $\gamma^\omega(1)=\gamma^\sigma(1)\cdot(-1)$ and therefore $B^\omega=B^\sigma \cdot (-1)$. Moreover, $\gamma_2^{\tilde{\omega}}(1)\cdot A^{\tilde{\omega}}=\gamma_2^{\tilde{\sigma}}(1)\cdot A^{\tilde{\sigma}}\cdot(-1)$ can be shown similar to the proof of i) by splitting the proof into the same three steps.
\end{proof}
For the inverse of $F$, we define a mapping
\begin{align*}
	G\colon\Spin(M,g)^2\to\Spin(M,g)^1\times_{\mathbb{Z}_2}P
\end{align*}
by the following. Let $a\in\Spin(M,g)^2_x$, $x\in M$.
\begin{enumerate}
	\item Choose a path $\omega\colon[0,1]\to M$ s.t. $\omega(0)=x_0$ and $\omega(1)=x$. Denote by $\gamma^\omega\colon[0,1]\to P$ the unique lift of $\omega$ to $P$ with $\gamma^\omega(0)=y_0$.
	\item Choose a lift $\tilde{\omega}\colon[0,1]\to\SO(M,g)$ of $\omega$ to $\SO(M,g)$.
	\item For $i=1,2$ choose lifts $\gamma_i^{\tilde{\omega}}\colon [0,1]\to\Spin(M,g)^i$ with 
\[\gamma_1^{\tilde{\omega}}(0)\cong\gamma_2^{\tilde{\omega}}(0).\]
	\item Denote by $A\in\Spin(m)$ the unique element of $\Spin(m)$ s.t.
	\[\gamma_2^{\tilde{\omega}}(1)\cdot A=a.\] 
\end{enumerate}
Then we define
\[G(a):=[\gamma_1^{\tilde{\omega}}(1)\cdot A,\gamma^\omega(1)].\]
Using the same ideas as above one can show that $G$ is well-defined. Directly from the definitions of $F$ and $G$ we get $F\circ G=id$ and $G\circ F=id$.

\section{Existence of orientation reversing simple closed geodesics} 
In the proof of Theorem \ref{thm 1} on page \pageref{page} we used the existence of an orientation reversing simple closed geodesic $\gamma\colon S^1\to N$ where $N$ is a non-orientable closed Riemannian manifold.

Starting with any closed curve $\gamma_0\colon S^1\to N$ it is a standard result that one can find a closed geodesic in the homotopy class of $\gamma_0$. A direct proof can be found e.g. in \cite[Theorem 1.5.1]{RiemGeomandGeomAna} or \cite[2.98 Theorem on p. 94]{GHL} and a proof using the heat flow method is given in \cite[Theorem 1.6.1]{RiemGeomandGeomAna}. This geodesic is orientation reversing provided that $\gamma_0$ is orientation reversing, but not necessarily without self-intersections. 

Moreover, it is well known that if $\pi_1(N)\neq \{1\}$, then there exists a closed geodesic on $N$ which minimizes length in the class of homotopically non-trivial closed curves on $N$ and this geodesic has no self-intersections, see e.g.  \cite[Lemma 1.5. (2) and Exercise 3 on p. 197]{Takashi}. However, this geodesic is not necessarily orientation reversing.

We prove the following lemma.

\begin{lemma}\label{lemma app1}Let $N$ be a closed non-orientable Riemannian manifold (in particular this implies that $\pi_1(N)\neq \{1\}$). Then there exists an orientation reversing simple closed geodesic $\gamma\colon S^1\to N$.
\end{lemma}
\begin{proof} Let $\gamma_0\colon S^1\to N$ be an orientation reversing closed geodesic. If $\gamma_0$ has no self-intersections, we are done. So assume that $\gamma_0$ is not injective. Then we can split $\gamma_0$ into two geodesic loops\footnote{A \textit{geodesic loop} is a geodesic $c\colon[0,1]\to N$ with $c(0)=c(1)$.} $\tilde{\gamma}_0, \hat{\gamma}_0\colon [0,1]\to N$ both based at the same point where $\tilde{\gamma}_0$ is orientation reversing (hence non-trivial). Let $c_0\colon S^1\to N$ be a smooth approximation of $\tilde{\gamma}_0$ homotopic to $\tilde{\gamma}_0$ with
	\[|L(\tilde{\gamma}_0)-L(c_0)|<\varepsilon\]
where $\varepsilon>0$ is small and $L$ denotes the length. Then there exists a closed orientation reversing geodesic $\gamma_1\colon S^1\to N$ in the homotopy class of $c_0$ which also minimizes length in its homotopy class, see e.g. \cite[Lemma 1.5. (1) on p. 197]{Takashi}. If $\gamma_1$ is injective, we are done. If not, we repeat the above process with $\gamma_0$ replaced by $\gamma_1$.

We have to ensure that this process stops after finitely many steps. This follows from the following two observations. Firstly, each $\gamma_k$ has positive length, i.e., $L(\gamma_k)>0$. Secondly, in each step, the length drops a fixed amount. To see the latter, we recall the following: if $N$ is a closed Riemannian manifold and $c$ is an arbitrary geodesic loop in $N$ (the base point is allowed to vary), then the length of $c$ is bounded from below by two times the injectivity radius of $N$,
\begin{align*}
	L(c)\ge 2\textup{inj}(N)=:C.
\end{align*}
Returning to the beginning of the proof, we choose $\varepsilon=\frac{1}{2}C$ to deduce
\begin{align*}
	L(\gamma_0)&=L(\tilde{\gamma}_0)+L(\hat{\gamma}_0)\\
	&\ge L(\tilde{\gamma}_0) + C\\
	&\ge L(c_0)+\frac12 C\\
	&\ge L(\gamma_1) + \frac12 C\\
\end{align*}
and entirely analogous
\[L(\gamma_{k+1})\le L(\gamma_k)-\frac12 C.\]
Hence in each step the length drops by at least $\frac12 C$.

\end{proof}

\end{appendices}

\newpage
\bibliographystyle{abbrv}
\bibliography{references}
\end{document}